\documentclass[12pt]{amsart}

\usepackage{amsmath,amsthm,amssymb,amsfonts,stmaryrd}

\usepackage{graphicx}
\usepackage{epstopdf}
\input xy
\xyoption{all}

\usepackage{fullpage}

\usepackage{hyperref}

\theoremstyle{plain}
\newtheorem{thm}{Theorem}[section]
\newtheorem{prp}[thm]{Proposition}
\newtheorem{lem}[thm]{Lemma}

\theoremstyle{definition}
\newtheorem{dfn}[thm]{Definition}

\theoremstyle{remark}
\newtheorem{rmk}[thm]{Remark}

\newcommand{\Hom}{\mathrm{Hom}}

\newcommand{\Aut}{\mathrm{Aut}}

\newcommand{\Out}{\mathrm{Out}}

\newcommand{\id}{\mathrm{id}}

\newcommand{\inj}{\hookrightarrow}

\newcommand{\op}{\mathrm{op}}

\newcommand{\GL}{\mathrm{GL}}

\newcommand{\SL}{\mathrm{SL}}

\newcommand{\Sym}{\mathrm{Sym}}

\newcommand{\la}{\langle}
\newcommand{\ra}{\rangle}

\newcommand{\wt}{\widetilde}

\newcommand{\ol}{\overline}

\newcommand{\CF}{\mathcal{F}}

\newcommand{\BF}{\mathbb{F}}

\newcommand{\BZ}{\mathbb{Z}}

\newcommand{\F}{\mathcal{F}}

\begin{document}
\title[Realizing Ruiz-Viruel exotic systems]{Minimal characteristic bisets and finite groups realizing Ruiz-Viruel exotic fusion systems} 
\author{Sejong Park}
\address{Institute of Mathematics, University of Aberdeen, AB24 3UE, UK}
\email{s.park@abdn.ac.uk}
\date{November 2010}
\maketitle
\begin{abstract}
Continuing our previous work~\cite{Park2010Realizing}, we determine a minimal left characteristic biset $X$ for every exotic fusion system $\CF$ on the extraspecial group $S$ of order $7^3$ and exponent $7$ discovered by Ruiz and Viruel~\cite{RuizViruel2004}, and analyze the finite group $G$ obtained from $X$ by the method of \cite{Park2010Realizing} which realizes $\CF$ as the full subcategory of the $7$-fusion system of $G$.  In particular, we obtain an upper bound for the exoticity index for $\CF$.
\end{abstract}

\section{Introduction}

In our previous work~\cite{Park2010Realizing}, we observed that every saturated fusion system $\CF$ on a finite $p$-group $S$ can be {\em realized} by a finite group $G$ containing $S$ as a subgroup, in the sense that $\CF = \CF_S(G)$, i.e.\ the fusion system on $S$ whose morphisms are the $G$-conjugation maps among subgroups of $S$.  This observation lead us to propose the notion of the {\em exoticity index} $e(\CF)$ of $\CF$ as the minimum of the values $\log_p |S_0:S|$ where $S_0$ is a Sylow $p$-subgroup of $G$ containing $S$ as $G$ runs over all finite groups such that $\CF = \CF_S(G)$.  With this definition, $\CF$ is exotic if and only if $e(\CF) >0$, and $e(\CF)$ measures how far the fusion system $\CF$ is from being the fusion system of some finite group.  In~\cite{Park2010Realizing} the finite group $G$ was constructed by using a finite $S$-$S$-biset $X$ called a left characteristic biset of the fusion system $\CF$ as $G = \Aut({}_{1}X)$, i.e. the automorphism group of $X$ viewed as a right $S$-set ignoring the left $S$-action (or equivalently, restricting the left $S$-action to the trivial subgroup, whence the subscript $1$ on the left).  Then we have 
\[
	G \cong S \wr \Sigma_{e(X)} \quad\text{where } e(X) = |X|/|S|.
\]
The upper bound of the exoticity index $e(\CF)$ of $\CF$ obtained from this $G$ is as follows:
\[
	e(\CF) \leq (e(X)-1)\log_p |S| + \sum_{i\geq 1} \left\lfloor \frac{e(X)}{p^i} \right\rfloor \leq e(X)(1+\log_p |S|).
\]
As far as we know, no other method for constructing a finite group realizing a given saturated fusion system has been found.

In this paper, we determine minimal left characteristic bisets for some saturated fusion systems $\CF$ on the extraspecial group $S$ of order $p^3$ and exponent $p$ including all exotic fusion systems on $S$ discovered by Ruiz and Viruel~\cite{RuizViruel2004}, and thereby obtain finite groups realizing $\CF$ which are the smallest among those obtained by the method of \cite{Park2010Realizing} described in the previous paragraph.

\begin{thm} \label{T:main1}
Let $p$ be an odd prime and let $S$ be the extraspecial group of order $p^3$ and exponent $p$.  Let $\CF$ be a saturated fusion system on $S$ such that all subgroups of $S$ of order $p^2$ are $\CF$-radical.  Then there exists a unique minimal left characteristic biset X for $\CF$, and we have
\[
	e(X) = \frac{p^5-1}{p-1}|\Out_\CF(S)|.
\]
\end{thm}
A more detailed version of this theorem is stated and proved in \S\ref{S:minimal biset for RV} as Theorem~\ref{T:minimal biset for RV} and Remark~\ref{R:numerical rel}.

The upper bound for the exoticity index of a Ruiz-Viruel exotic fusion system $\CF$ given by the above theorem is quite large.  An extraspecial group $S$ of order $p^3$ and exponent $p$ affords exotic fusion systems only when $p=7$, and in that case there are three of them with the order of the outer automorphism groups of $S$ being 48, 72 and 96, respectively.  The upper bound that we obtain for the case where $|\Out_\CF(S)|=48$ is 425744, for example.

A natural question that follows is whether we can cut down the group $G$ to get a smaller upper bound for the exoticity index.  By Alperin's fusion theorem, an obvious candidate is the subgroup $H$ of $G$ generated by the normalizers of the $\CF$-essential subgroups of $S$.  After close analysis of some elements of the normalizers of $\CF$-essential subgroups of $S$, we find that even $H$ is quite large, though this gives us a better understanding of the fusion action of $G$ on $S$.

\begin{thm} \label{T:main2}
Let $p$ be an odd prime and let $S$ be an extraspecial group of order $p^3$ and exponent $p$.  Let $\CF$ be a saturated fusion system on $S$ such that all subgroups of $S$ of order $p^2$ are $\CF$-radical.  Let $X$ be the minimal left characteristic biset $X$ for $\CF$ given by Theorem~\ref{T:main1}.  Let $G = \Aut({}_{1}X) \cong S \wr \Sigma_{e(X)}$ and let $H$ be the subgroup of $G$ generated by the normalizers of the $\CF$-essential subgroups of $S$.  Then the image of $H$ in $\Sigma_{e(X)}$ is a transitive subgroup of $\Sigma_{e(X)}$.
\end{thm}

This theorem is proved in \S\ref{S:finite group for RV} as Theorem~\ref{T:main2'}.

{\bf Organization of the paper:}  In \S\ref{S:char biset}, we recall the definition of characteristic bisets for fusion systems and fix basic notations.  We also compare notational conventions in the literature about bisets concerning fusion systems in Remark~\ref{R:opposite}.  In \S\ref{S:fixed pts}, we introduce the notion of  layers of bisets and review some basic results about the number of fixed points which will be used in the subsequent sections.  \S\ref{S:min biset up to index p} contains a general result which determines the top two layers of a characteristic biset of any saturated fusion system.  Using this result and detailed information of the saturated fusion systems on an extraspecial group $S$ of order $p^3$ and exponent $p$ in \cite{RuizViruel2004}, we determine minimal left characteristic bisets for saturated fusion systems $\CF$ on $S$ such that all subgroups of $S$ of order $p^2$ are $\CF$-radical in \S\ref{S:minimal biset for RV}.  A more detailed version of Theorem~\ref{T:main1} is stated and proved in this section as Theorem~\ref {T:minimal biset for RV} and Remark~\ref{R:numerical rel}.  In \S\ref{S:finite group for RV}, we analyze some elements of the normalizers of $\CF$-essential subgroups of $S$ and obtain Theorem~\ref{T:main2} as a consequence.  Finally in the Appendix~\ref{S:coeff char id}, we determine coefficients of characteristic idempotents as an easy application of results in \S\ref{S:min biset up to index p} and \S\ref{S:minimal biset for RV}.

{\bf Acknowledgments:} The author would like to express his special thanks to Ronald Solomon and Justin Lynd for very helpful discussions about the content of this paper and their warm hospitality whenever he visited The Ohio State University.

\section{Characteristic bisets for fusion systems} \label{S:char biset}

First let us fix notations.  We keep the notations in our previous paper~\cite{Park2010Realizing} and expand on them as needed.  Our notations are largely compatible with (and in fact many of them derive from) those in \cite{Ragnarsson2006,RagnarssonStancuSFSasIdemp}, except that the directions of maps are inverted.  See Remark~\ref{R:opposite} for more details.  

Let $S_1$, $S_2$ be groups. An $S_1$-$S_2$-biset is a set $X$ with left $S_1$-action and right $S_2$-action such that $(ux)v=u(xv)$ for $x\in X$, $u\in S_1, v\in S_2$.  An $S_1$-$S_2$-biset $X$ can be viewed as a (left) $(S_1\times S_2)$-set via $(u,v)\cdot x = uxv^{-1}$ for $x\in X$, $u\in S_1$, $v\in S_2$ and vice versa.  Throughout this paper, we will go freely back and forth between $S_1$-$S_2$-bisets and $(S_1\times S_2)$-sets using this correspondence.  That way, we can exploit the advantages of both perspectives.  On one hand, bisets can be composed: if $S_3$ is another group and $Y$ is an $S_2$-$S_3$-biset, the set 
\[
	X\times_{S_2} Y = (X\times Y) / \sim,
\]
where $(xu,y) \sim (x,uy)$ for $x \in X$, $y \in Y$, $u \in S_2$, is an $S_1$-$S_3$-biset by the left action of $S_1$ on $X$ and the right action of $S_3$ on $Y$. On the other hand, $X$ viewed as an $(S_1 \times S_2)$-set has the fixed-point subset $X^H = \{ x\in X \mid hx=x \text{ for all } h \in H \}$ by any subgroup $H\leq S_1 \times S_2$, and the number of fixed-points $|X^H|$ as $H$ runs over all subgroups of $S_1\times S_2$ determine the isomorphism type of $X$ by Burnside's theorem.

For a subgroup $Q$ of $S_1$ and an injective group homomorphism $\varphi \colon Q \to S_2$, let
\[
	S_1 \times_{(Q,\varphi)}S_2 = (S_1 \times S_2) / \sim
\]
where $(xu,y) \sim (x,\varphi(u)y)$ for $x\in S_1$, $y\in S_2$, $u\in Q$, and let $\la x,y \ra$ be the $\sim$-equivalence class containing $(x,y)$.  One can view this set as an $S_1$-$S_2$-biset by $u \cdot \la x,y\ra = \la ux,y\ra$, $\la x,y\ra \cdot v = \la x,yv\ra$ for $x, u \in S_1$, $y,v \in S_2$.  It is a transitive $S_1$-$S_2$-biset with free left $S_1$-action and free right $S_2$-action.  In fact, every such $S_1$-$S_2$-biset is of this form.  Viewed as an $(S_1\times S_2)$-set, it is isomorphic to 
\[
	(S_1\times S_2)/\Delta^{\varphi}_{Q}
\]
where $\Delta^{\varphi}_{Q} = \{ (u,\varphi(u)) \colon u\in Q \}$. 

Two transitive $S_1$-$S_2$-bisets $S_1\times_{(Q,\varphi)} S_2$ and $S_1\times_{(R,\psi)} S_2$ where $Q, R \leq S_1$, $\varphi\colon Q \to S_2$, $\psi\colon R \to S_2$ are isomorphic if and only if $\Delta_{Q}^{\varphi}$ is conjugate to $\Delta_{R}^{\psi}$ in $S_1\times S_2$, which in turn holds if and only if there exist $x\in S_1$, $y\in S_2$ such that ${}^x{}Q:=xQx^{-1} =R$ and $\psi\circ c_x|_Q = c_y \circ \varphi$ where $c_x \colon S_1\to S_1$ denotes the conjugation map by $x$ given by $c_x(u)=xux^{-1}$ for $u\in S_1$, and $c_y$ is defined analogously.  If this is the case, we say that $\varphi$ and $\psi$ are {\em $S_1$-$S_2$-conjugate}, and write $\varphi \sim_{(S_1,S_2)} \psi$.  We write $[\varphi]$ for the $S_1$-$S_2$-conjugacy class containing $\varphi$.  More generally, we say that $\varphi$ is {\em $S_1$-$S_2$-subconjugate} to $\psi$ and write $\varphi \precsim_{(S_1,S_2)} \psi$ if there exist $x\in S_1$, $y\in S_2$ such that ${}^x{}Q \leq R$ and $\psi\circ c_x|_Q = c_y \circ \varphi$

Finally, for an $S_1$-$S_2$-biset $X$, $Q\leq S_1$, and an injective group homomorphism $\varphi\colon Q\to S_2$, let ${}_{\varphi}X$ be the $Q$-$S_2$-biset obtained from $X$ where the left $Q$-action is induced by $\varphi$; let ${}_{Q}X := {}_{(\id_Q)}X$.  Equivalently, 
\[
	{}_{\varphi}X \cong (Q \times_{(Q,\varphi)} S_1) \times_{S_1} X.
\]
For $R\leq S_2$ and an injective group homomorphism $\psi\colon R \to S_2$, $S_1$-$R$-bisets $X_\psi$ and $X_R$ are defined analogously.

\begin{dfn} \label{D:stable biset}
Let $\CF$ be a fusion system on a finite $p$-group $S$. A finite $S$-$S$-biset $X$ is called a {\em left characteristic biset} for the fusion system $\CF$ if it satisfies the following conditions:
\begin{enumerate}
\renewcommand{\theenumi}{\Alph{enumi}}
\item Every transitive subbiset of $X$ is of the form $S\times_{(Q,\varphi)} S$ where $Q\leq S$ and $\varphi\in\Hom_\CF (Q, S)$. \label{generation}
\item $e(X):=|X|/|S| \not\equiv 0 \mod p$. \label{coprime}
\item (left $\CF$-stability) ${}_{Q}X \cong {}_{\varphi}X$ as $Q$-$S$-bisets for every $Q\leq S$ and $\varphi\in\Hom_\CF(Q,S)$. \label{stable}
\end{enumerate}
Similarly, $X$ is called a {\em right characteristic biset} for $\CF$ if it satisfies the conditions \eqref{generation}, \eqref{coprime}, and
\begin{enumerate}
\renewcommand{\theenumi}{\Alph{enumi}'}
\addtocounter{enumi}{2}
\item (right $\CF$-stability) $X_Q \cong X_\varphi$ as $S$-$Q$-bisets for every $Q\leq S$ and $\varphi\in\Hom_\CF(Q,S)$. \label{rstable}
\end{enumerate}
If X satisfies the conditions \eqref{generation}, \eqref{coprime}, \eqref{stable} and \eqref{rstable}, then $X$ is called a {\em characteristic biset} for $\CF$.
\end{dfn}

Throughout this paper, we will refer to the conditions in this definition by the symbols \eqref{generation}, \eqref{coprime}, \eqref{stable} and \eqref{rstable}.

The following result, due to Broto, Levi and Oliver, is fundamental for this paper.

\begin{thm}[{\cite[5.5]{BrotoLeviOliver2003b}}]
Every saturated fusion system has a characteristic biset.
\end{thm}

\begin{rmk}
\cite[5.5]{BrotoLeviOliver2003b} only states the existence of a right characteristic biset. But in fact what is constructed in the proof of \cite[5.5]{BrotoLeviOliver2003b} is a characteristic biset, as pointed out by Ragnarsson in the paragraph right before Proposition 4.4 in \cite{Ragnarsson2006}.
\end{rmk}

We define an object which will play a central role in this paper.

\begin{dfn}
Let $\CF$ be a saturated fusion system on a finite $p$-group $S$. A (left or right) characteristic biset $X$ for $\CF$ is called {\em minimal} if it has the smallest size among all possible (left or right) characteristic bisets for $\CF$.
\end{dfn}

Finally, we introduce the opposite operation, which is useful when comparing left and right stability.

\begin{dfn}
Let $S_1$ and $S_2$ be groups. For an $S_1$-$S_2$-biset $X$, let $X^\op$ be the $S_2$-$S_1$-biset which has the same underlying set as $X$ and such that $v\cdot x\cdot u = u^{-1}xv^{-1}$ for $x\in X$, $u\in S_1$, $v \in S_2$.
\end{dfn}

If $X$ is a left (resp.\ right) characteristic biset for a fusion system $\CF$, then $X^\op$ is a right (resp.\ left) characteristic biset for $\CF$.  If $S_1$, $S_2$ are groups, $Q\leq S_1$ and $\varphi\colon Q\to S_2$ is an injective group homomorphism, then
\[
	( S_1 \times_{(Q,\varphi)} S_2 )^\op \cong S_2 \times_{(\varphi(Q),\varphi^{-1})} S_1
\]
See \cite[\S 3.5]{RagnarssonStancuSFSasIdemp} for a more discussion of this opposite operation.

\begin{rmk} \label{R:opposite}
Our use of the terminology `characteristic biset' derives from Ragnarsson~\cite{Ragnarsson2006}, and Ragnarsson and Stancu~\cite{RagnarssonStancuSFSasIdemp}, whereas Puig~\cite[Ch.\ 21]{PuigBook2009} calls them {\em basic sets}.   But our convention in which maps defining transitive bisets are directed from left to right is close to Puig's; Ragnarsson and Stancu use the opposite notations.  More precisely, the correspondence between Ragnarsson and Stancu's notation and ours is as follows:
\[
	[S_1 \times_{(Q,\varphi)}S_2] = [\varphi(Q),\varphi^{-1}]_{S_2}^{S_1} = \left([Q,\varphi]_{S_1}^{S_2}\right)^\op,
\]
where the square bracket denotes the isomorphism class of bisets.  Accordingly, our left (resp.\ right) stabilitiy corresponds to Ragnarsson and Stancu's right (resp.\ left) stability.
\end{rmk}

\section{Layers of characteristic bisets and counting fixed-points} \label{S:fixed pts}

\begin{dfn}
Let $S$ be a finite group. Let $X$ be an $S$-$S$-biset with free left and right action of $S$.  For each $r\geq 0$, let $X_r$ denote the `$r$-th layer' of $X$, i.e.\ the union of the transitive subbisets of $X$ of the form
\[
	S\times_{(Q,\varphi)}S\quad\text{with } |S:Q|=p^r,
\]
so that
\[
	X = X_0 \amalg X_1 \amalg X_2 \amalg \cdots.
\]
Let $X_{\leq r}=\coprod_{0\leq i\leq r} X_i$.  
\end{dfn}

A very effective way of studying finite sets with action of a finite group is counting the number of fixed-points by subgroups.  Indeed it is a standard fact due to Burnside that, for a finite group $G$, two finite $G$-sets $X$ and $Y$ are isomorphic if and only if $|X^H|=|Y^H|$ for all $H\leq G$.  

First we make a simple, yet useful, observation.

\begin{lem} \label{T:fixed points lemma}
Let $\CF$ be a saturated fusion system on a finite $p$-group $S$ with a left or right characteristic biset $X$.
\begin{enumerate}
\item For $U\leq S\times S$, we have $X^{U}\neq\emptyset$ if and only if there is a transitive subbiset $S\times_{(Q,\varphi)} S$ of $X$ such that $U$ is conjugate in $S\times S$ to a subgroup of $\Delta_Q^{\varphi}$.  In particular, $U=\Delta_{R}^{\psi}$ for some $R\leq S$ and $\psi\in\Hom_\CF(R,S)$.
\item $|X^{\Delta_{Q}^{\varphi}}|=|X_{\leq r}^{\Delta_{Q}^{\varphi}}|$ if $|S\colon Q|=p^r$.
\end{enumerate}
\end{lem}

\begin{proof}
In general, if $G$ is a group and $H$, $K$ are subgroups of $G$, we have $(G/H)^K \neq \emptyset$ if and only if $K$ is $G$-conjugate to a subgroup of $H$.  The lemma follows immediately from this fact and that the family of subgroups of $S\times S$ of the form $\Delta_Q^\varphi$ where $Q\leq S$ and $\varphi\in\Hom_\CF(Q,S)$ is invariant under conjugation in $S\times S$.
\end{proof}

In the following lemma, we record what the stability conditions \eqref{stable} and \eqref{rstable} in Definition~\ref{D:stable biset} mean in terms of the number of fixed-points.  It is a slight variation of \cite[4.8]{RagnarssonStancuSFSasIdemp}, which reduces the number of equations to consider.

\begin{lem} \label{T:stability by fixed-pts}
Let $\CF$ be a fusion system on a finite $p$-group $S$, and let $X$ be a finite $S$-$S$-biset satisfying the condition \eqref{generation}. Then the left stability condition \eqref{stable} is equivalent to the following condition:
\begin{enumerate}
\item[] $|X^{\Delta_{Q}^{\varphi}}|=|X^{\Delta_{\varphi(Q)}^{\id}}|$ for all $Q\leq S$, $\varphi\in
\Hom_\CF(Q,S)$;
\end{enumerate}
the right stability condition \eqref{rstable} is equivalent to the following condition:
\begin{enumerate}
\item[] $|X^{\Delta_{Q}^{\varphi}}|=|X^{\Delta_{Q}^{\id}}|$ for all $Q\leq S$, $\varphi\in\Hom_
\CF(Q,S)$.
\end{enumerate}
\end{lem}

\begin{proof}
We consider the condition \eqref{stable}, the case for \eqref{rstable} being analogous.  We have ${}_{Q}X \cong {}_{\varphi}X$ if and only if $|({}_{Q}X)^H| = |({}_{\varphi}X)^H|$ for all $H\leq S\times S$.  Since $X$ satisfies the condition \eqref{generation}, Lemma~\ref{T:fixed points lemma} implies that both $({}_{Q}X)^H$ and $({}_{\varphi}X)^H$ will be empty unless $H = \Delta_R^\psi$ for some $R\leq Q$ and $\psi\in\Hom_\CF(R,S)$.  Now $({}_{Q}X)^{\Delta_R^\psi} = X^{\Delta_R^\psi}$, while $({}_{\varphi}X)^{\Delta_R^\psi} = X^{\Delta_{\varphi(R)}^{\psi\circ\varphi^{-1}}}$.  Thus ${}_{Q}X \cong {}_{\varphi}X$ if and only if $|X^{\Delta_R^\psi}| = |X^{\Delta_{\varphi(R)}^{\psi\circ\varphi^{-1}}}|$ for all $R\leq Q$, $\psi\in\Hom_\CF(R,S)$.  The latter condition implies in particular (when $R=Q$ and $\psi=\varphi$) $|X^{\Delta_Q^\varphi}| = |X^{\Delta_{\varphi(Q)}^{\id}}|$.  Thus the condition \eqref{stable} implies that $|X^{\Delta_{Q}^{\varphi}}|=|X^{\Delta_{\varphi(Q)}^{\id}}|$ for all $Q\leq S$, $\varphi\in
\Hom_\CF(Q,S)$.  Conversely, if $|X^{\Delta_Q^\varphi}| = |X^{\Delta_{\varphi(Q)}^{\id}}|$ for all $Q\leq S$, $\varphi\in
\Hom_\CF(Q,S)$, then the condition \eqref{stable} holds, because
\[
	|X^{\Delta_R^\psi}| = |X^{\Delta_{\psi(R)}^{\id}}| = |X^{\Delta_{\varphi(R)}^{\psi\circ\varphi^{-1}}}|
\]
for all $R\leq Q\leq S$, $\psi\in\Hom_\CF(R,S)$, $\varphi\in\Hom_\CF(Q,S)$.  
\end{proof}

We recall the formula for the number of fixed-points of transitive bisets.

\begin{lem}[{\cite[3.10]{RagnarssonStancuSFSasIdemp}}] \label{T:no of fixed pts}
Let $S$ be a finite group.  Let $Q, R\leq S$, and let $\varphi\colon Q\to S$, $\psi\colon R\to S$ be injective group homomorphisms.  Then
\[
	|((S\times S)/\Delta_{Q}^{\varphi})^{\Delta_{R}^{\psi}}| = \frac{|N_{\psi,\varphi}|}{|Q|}|C_S(\psi(R))|
\]
where 
\[
	N_{\psi,\varphi} = \{ x\in S\mid {}^{x}R \leq Q, \exists y\in S\colon \varphi\circ c_x|_R = c_y \circ \psi \}.
\]
\end{lem}

Note that $N_{\psi,\varphi} = \emptyset$ unless $\psi \precsim_{(S,S)} \varphi$.  Also, $N_{\varphi,\varphi}=N_\varphi$, the largest subgroup of $N_S(Q)$ to which $\varphi$ can be extended.  In the special case where $Q=S$, Lemma~\ref{T:no of fixed pts} takes the following form.

\begin{lem} \label{T:no of fixed pts on top piece}
Let $S$ be a finite group.  Let $R\leq S$, and let $\alpha\colon S\to S$, $\psi\colon R\to S$ be injective group homomorphisms.  Then
\[
	|((S\times S)/\Delta_{S}^{\alpha})^{\Delta_{R}^{\psi}}| =
	\begin{cases}
		|C_S(\psi(R)|,	&\text{if $\psi \precsim_{(S,S)} \alpha$},\\
		0,				&\text{otherwise}.
	\end{cases}
\] 
\end{lem}

\begin{proof}
By Lemma~\ref{T:no of fixed pts}, $((S\times S)/\Delta_{S}^{\alpha})^{\Delta_{R}^{\psi}} \neq\emptyset$ if and only if $N_{\psi,\alpha}\neq\emptyset$, i.e.\ if there are $x,y\in S$ such that $\alpha\circ c_x|_{R} = c_y\circ\psi$.  Since $\alpha\circ c_x = c_{\alpha(x)}\circ\alpha$, we have $N_{\psi,\alpha}\neq\emptyset$ if and only if $N_{\psi,\alpha}=S$.
\end{proof}

\section{Minimal characteristic bisets up to index $p$ subgroups} \label{S:min biset up to index p}

In this section, we determine the coefficients of the top two layers $X_0$ and $X_1$ of a left or right characteristic biset $X$ for a saturated fusion system $\CF$.  It is well-known that $X_0$ has the following form.

\begin{prp} \label{T:layer0}
Let $X$ be a left or right characteristic biset for a saturated fusion system $\CF$ on a finite $p$-group $S$. Then there is a positive integer $c_0$ which is not divisible by $p$ such that
\[
	X_0 \cong c_0 \coprod_{[\alpha]\in\Out_\CF(S)} (S\times S)/\Delta_{S}^{\alpha}.
\]
\end{prp}

\begin{proof}
Write $X_0 \cong \coprod_{[\alpha]\in\Out_\CF(S)} c_\alpha (S\times S)/\Delta_{S}^{\alpha}$ where the $c_\alpha
$ are nonnegative integers. We have $|X|/|S| \equiv |X_0|/|S| \equiv \sum_{\alpha}c_\alpha \mod p$.  Thus 
the condition \eqref{coprime} is equivalent to that $\sum_{\alpha}c_\alpha \not\equiv 0\mod p$.  In particular, there is $
\beta\in\Aut_\CF(S)$ such that $c_\beta>0$.  Lemma~\ref{T:stability by fixed-pts} implies that $|X^{\Delta_{S}^{\alpha}}|=|
X^{\Delta_{S}^{\beta}}|$ for every $\alpha\in\Aut_\CF(S)$.  But $|X^{\Delta_{S}^{\alpha}}|=|
X_0^{\Delta_{S}^{\alpha}}|=c_\alpha|Z(S)|$ by Lemma~\ref{T:fixed points lemma} and Lemma~\ref{T:no of fixed pts}.  Thus we have $c_\alpha=c_\beta =: c_0$ for every $\beta\in\Aut_\CF(S)$.  Then $|X|/|S| \equiv \sum_{\alpha}c_\alpha \equiv c_0|\Out_\CF(S)| \not\equiv 0 \mod p$.  It follows that $c_0\not\equiv 0\mod p$.
\end{proof}

To determine $X_1$, we need the following definitions.

\begin{dfn}
Let $\CF$ be a fusion system on a finite $p$-group $S$.  Let $Q\leq R\leq S$ and let $\varphi\in\Aut_\CF(Q,S)$.  We say that $\varphi$ is {\em extendable (in $\CF$) to 
$R$} if there is a morphism $\psi\in\Hom_\CF(R,S)$ such that $\varphi=\psi|_Q$.  We say that $\varphi$ is 
{\em extendable} if it is extendable to some subgroup of $S$ properly containing $Q$; otherwise, we say 
that $\varphi$ is {\em nonextendable}. Let
\[
	\Hom_\CF^{e,R}(Q,S) = \{ \varphi\in\Hom_\CF(Q,S) \mid \text{ $\varphi$ is extendable to $R$} \},
\]
\[
	\Hom_\CF^{e}(Q,S) = \{ \varphi\in\Hom_\CF(Q,S) \mid \text{ $\varphi$ is extendable } \},
\]
\[
\Hom_\CF^{n}(Q,S) = \{ \varphi\in\Hom_\CF(Q,S) \mid\text{ $\varphi$ is nonextendable } \}.
\]
\end{dfn}

Note that $\Hom_\CF^{e}(Q,S)=\bigcup_{Q<R\leq S} \Hom_\CF^{e,R}(Q,S)$.  Also, if $|S:Q|=p$, then $\Hom_
\CF^{e,S}(Q,S)=\Hom_\CF^{e}(Q,S)$.

The following proposition is implicit in \cite[5.5]{RagnarssonStancuSFSasIdemp}.

\begin{prp}
Let $X$ be a left or right characteristic biset for a saturated fusion system $\CF$ on a finite $p$-group $S$.  Let $\varphi\in\Hom_\CF^{n}(Q,S)$ for some $Q\leq S$.  Then $X$ has a transitive subbiset isomorphic to $S\times_{(Q,\varphi)} S$.
\end{prp}

\begin{proof}
Assume that $X$ is a left characteristic biset for $\CF$; the proof is identical when $X$ is a right characteristic biset.  By Lemma~\ref{T:stability by fixed-pts}, we have $|X^{\Delta_Q^\varphi}|=|X^{\Delta_{\varphi(Q)}^\id}|$.  By Proposition~\ref{T:layer0} and Lemma~\ref{T:no of fixed pts on top piece}, we have $|X^{\Delta_{\varphi(Q)}^\id}| \geq |((S\times S)/\Delta_S^{\id})^{\Delta_{\varphi(Q)}^\id}| > 0$.  Thus we have $|X^{\Delta_Q^\varphi}| \neq 0$.  Lemma~\ref{T:fixed points lemma} and the assumption that $\varphi$ is nonextendable,  it follows that $X$ has a transitive subbiset isomorphic to $S\times_{(Q,\varphi)}S$.
\end{proof}

We denote by $\Hom_\CF(Q,S)/\sim_{(S,S)}$ the set of $\sim_{(S,S)}$-equivalence classes of 
elements of $\Hom_\CF(Q,S)$.  For instance, we have $\Out_\CF(S)=\Aut_\CF(S)/\sim_{(S,S)}$.  The following proposition is based on an important property of $\CF$-centric subgroups which can be found in \cite[A.8]{BrotoLeviOliver2003b}.  Both the proposition and the proof are crucial for many results in this paper.

\begin{prp} \label{restriction isom}
Let $\CF$ be a saturated fusion system on a finite $p$-group $S$.  Let $Q\leq R\leq S$ be normal subgroups of $S$ which are $\CF$-centric.  Then restriction to $Q$ induces a bijection
\[
	\Hom_\CF(R,S)/\sim_{(S,S)} \xrightarrow{\cong} \Hom_{\CF}^{e,R}(Q,S)/\sim_{(S,S)}.
\]
In particular, if $|S:Q|=p$, then we have
\[
	\Out_\CF(S) \xrightarrow{\cong} \Hom_{\CF}^{e}(Q,S)/\sim_{(S,S)}
\]
\end{prp}

\begin{proof}
The given map is well-defined because $Q\unlhd S$.  It is clearly surjective.  To show that it is injective, let $\varphi, \psi \in \Hom_\CF(R,S)$ and suppose $\varphi|_Q \sim_{(S,S)} \psi|_Q$.  That means there exist $x \in N_S(Q)$ and $y\in S$ such that $c_y\circ\psi|_Q=\varphi\circ c_x|_Q$.  Since $Q$ is $\CF$-centric, we apply \cite[A.8]{BrotoLeviOliver2003b} to $c_y\circ\psi|_R$ and $\varphi\circ c_x|_R$ (the latter is well-defined because $R\unlhd S$) and get $z\in Z(Q)$ such that $c_y\circ\psi=\varphi\circ c_x\circ c_z|_R$.  In other words, $\varphi \sim_{(S,S)} \psi$, showing that the given map is also injective.
\end{proof}

Now we determine $X_1$ of a characteristic biset $X$ for a saturated fusion system, which amounts to solving a part of the system of linear equations given in \cite[5.6]{Ragnarsson2006}.

\begin{prp} \label{T:layer1}
Let $\CF$ be a saturated fusion system on a finite $p$-group $S$.  Let $\{ P_i \mid 1\leq i \leq n \}$ be the set of subgroups of $S$ of index $p$.  For each $1\leq i\leq n$, let $\{ \varphi_{i,j}\colon 
P_i \to S \mid 0\leq j\leq m_i \}$ be a set of representatives of $\sim_{(S,S)}$-equivalence classes of $\CF$-morphisms with domain $P_i$ such that $\varphi_{i,0}=\iota_{P_i}$, the inclusion map $P_i\hookrightarrow S$.  Let $X$ be a finite $S$-$S$-biset with
\[
	X_0 \cong c_0 \coprod_{[\alpha]\in\Out_\CF(S)} (S\times S)/\Delta_{S}^{\alpha}
\]
for some positive integer $c_0$ which is not divisible by $p$.  
\begin{enumerate}
\item If $X$ is a right characteristic biset for $\CF$, then there is an integer $c_1^{(i)}\geq 0$ for each $1\leq i\leq n$ such that
\[
	X_1 \cong \coprod_{\substack{ 1\leq i, j\leq n \\ \varphi_{i,j} \text{ extendable}}} c_1^{(i)} (S\times S)/\Delta_{P_i}^{\varphi_{i,j}} \amalg 
		\coprod_{\substack{ 1\leq i, j\leq n \\ \varphi_{i,j} \text{ nonextendable}}}(c_0+pc_1^{(i)})(S\times S)/\Delta_{P_i}^{\varphi_{i,j}}.
\]
\item If $X$ is a left characteristic biset for $\CF$, then there is an integer $c_1^{(i)}\geq 0$ for each $1\leq i\leq n$ such that
\[
	X_1 \cong \coprod_{\substack{ 1\leq i, j\leq n \\ \varphi_{i,j} \text{ extendable}}} c_1^{(i)} (S\times S)/\Delta_{\varphi_{i,j}(P_i)}^{\varphi_{i,j}^{-1}} \amalg 
		\coprod_{\substack{ 1\leq i, j\leq n \\ \varphi_{i,j} \text{ nonextendable}}}(c_0+pc_1^{(i)}) (S\times S)/\Delta_{\varphi_{i,j}(P_i)}^{\varphi_{i,j}^{-1}}.
\]
\end{enumerate} 
\end{prp}

\begin{proof}
We prove (1) and obtain (2) by taking opposite of (1).  Suppose that $X$ is a right characteristic biset for $\CF$.  Write $X_1 \cong \coprod_{i,j}c_1^{(i,j)}(S\times S)/\Delta_{P_i}^{\varphi_{i,j}}$ where the $c_1^{(i,j)}$ are nonnegative integers.  For each pair $(i,j)$, we have $|X^{\Delta_{P_i}^{\varphi_{i,j}}}|=|X^{\Delta_{P_j}^{\varphi_{i,0}}}|$ by Lemma~\ref{T:stability by fixed-pts}.  Since $|S:P_i|=p$, we have $|X^{\Delta_{P_i}^{\varphi_{i,j}}}|=|X_{0}^{\Delta_{P_i}^{\varphi_{i,j}}}|+|X_{1}^{\Delta_{P_i}^{\varphi_{i,j}}}|$ by Lemma~\ref{T:fixed points lemma}, and
\begin{align*}
	|X_{0}^{\Delta_{P_i}^{\varphi_{i,j}}}|&=c_0a_{i,j}|C_S(\varphi_{i,j}(P_i))|,\\
	|X_{1}^{\Delta_{P_i}^{\varphi_{i,j}}}|&=c_1^{(i,j)}|((S\times S)/
\Delta_{P_i}^{\varphi_{i,j}})^{\Delta_{P_i}^{\varphi_{i,j}}}|=c_1^{(i,j)}\frac{|N_{\varphi_{i,j}}|}{|P_i|}|C_S(\varphi_{i,j}(P_i))|,
\end{align*}
where
\[
	a_{i,j}=|\{ [\alpha]\in\Out_\CF(S)\mid \alpha|_{P_i}\sim_{(S,S)}\varphi_{i,j} \}|,
\]
by Lemma~\ref{T:no of fixed pts} and Lemma~\ref{T:no of fixed pts on top piece}.  Thus we have
\[
	\left(c_0a_{i,j}+c_1^{(i,j)}|N_{\varphi_{i,j}}|/|P_i|\right)|C_S(\varphi_{i,j}(P_i))| = \left(c_0a_{i,0}+c_1^{(i,0)}|N_{\varphi_{i,0}}|/|P_i|\right)|C_S(P_i)|.
\]

Case 1: $\varphi_{i,j}$ is extendable.  Then $\varphi_{i,j}=\alpha|_{P_i}$ for some $\alpha\in\Aut_\CF(S)$.  Then $C_S(\varphi_{i,j}(P_i)) = C_S(\alpha(P_i)) = \alpha(C_S(P_i))$, so $|C_S(\varphi_{i,j}(P_i))| = |C_S(P_i)|$.  Also $N_{\varphi_{i,j}}=N_{\varphi_{i,0}}=S$.  Now if $\beta\in\Aut_\CF(S)$ and $\beta|_{P_i}\sim_{(S,S)}\varphi_{i,j}$, then $\alpha|_{P_i}\sim_{(S,S)}\beta|_{P_i}$, so $\beta^{-1}\circ\alpha|_{P_i}\sim_{(S,S)}\id_{P_i}$.  This shows that $a_{i,j}=a_{i,0}$.  Hence, we have $c_1^{(i,j)}=c_1^{(i,0)}=:c_1^{(i)}$.

Case 2: $\varphi_{i,j}$ is not extendable.  Then $P_i$ and $\varphi_{i,j}(P_i)$ are $\CF$-centric by Alperin's fusion theorem.  So $|C_S(\varphi_{i,j}(P_i))| = |Z(\varphi_{i,j}(P_i))| = |\varphi_{i,j}(Z(P_i))| = |Z(P_i)| = |C_S(P_i)|$.  Also $N_{\varphi_{i,j}}=P_i$ and $N_{\varphi_{i,0}} = S$.  We have $a_{i,j}=0$ since $\varphi_{i,j}$ is nonextendable.  On the other hand, suppose $\alpha|_{P_i}\sim_{(S,S)}\iota_{P_i}$, i.e.\@ there is $\alpha\in\Aut_\CF(S)$ such that $\alpha|_{P_i}=c_x|_{P_i}$ for some $x\in S$.  By \cite[A8]{BrotoLeviOliver2003b}, we have $\alpha=c_x\circ c_z$ for some $z\in Z(P_i)$, and so $\alpha\sim_{(S,S)}\id_S$.  Thus $a_{i,0}=1$.  Therefore, we have $c_1^{(i,j)}=c_0+pc_1^{(i)}$.
\end{proof}

\section{Minimal bisets for Ruiz-Viruel exotic systems} \label{S:minimal biset for RV}

In this section, we determine minimal left characteristic bisets for Ruiz-Viruel exotic fusion systems~\cite{RuizViruel2004}.  In particular, we show that every Ruiz-Viruel exotic fusion system has a unique minimal left characteristic biset, and that in fact it is also a unique minimal right characteristic biset.

First we review some of the results on the Ruiz-Viruel exotic fusion systems and fix notations.  Throughout this section, let $p$ be an odd prime, and let $S$ be an extraspecial $p$-group of order $p^3$ and exponent $p$. Then $S$ has  center $Z(S)=\la z \ra$ of order $p$, and exactly $p+1$ subgroups
\[
	V_i = \la z, u_i \ra\quad(0\leq i\leq p)
\]
of order $p^2$, which are elementary abelian and centric in $S$.  Therefore, for any fusion system $\CF$ on $S$, all the $V_i$ are $\CF$-centric.  Let us fix the ordered basis $(z,u_i)$ for $V_i$ as an $\mathbb{F}_p$-vector space.  Via this ordered basis we get an isomorphism $\Aut(V_i) \cong \GL_2(p)$.  Then $U_i:=\Aut_S(V_i) \cong \left\{\begin{pmatrix} 1 & *\\0 & 1 \end{pmatrix}\right\} =: U \leq \GL_2(p)$. Set some subgroups of $\GL_2(p)$ as
\[
	B:=N_{\GL_2(p)}(U)=\left\{\begin{pmatrix} * & *\\0 & * \end{pmatrix}\right\},\quad T:=\left\{\begin{pmatrix} * & 0\\0 & * 
\end{pmatrix}\right\},\quad R:=\left\{\begin{pmatrix} 0 & *\\ * & 0 \end{pmatrix}\right\},
\]
and denote by $B_i$, $T_i$ and $R_i$ their isomorphic images in $\Aut_\CF(V_i)$.  Note that $T \sqcup R$ is a set of representatives of $U$-$U$-double cosets in $\GL_2(p)$.

Let $\CF$ be a saturated fusion system on $S$.  By the extension axiom, we have
\begin{gather*}
	\Aut_\CF(V_i)\cap B_i = \Aut_\CF^{e}(V_i),\\
	\Aut_\CF(V_i) - B_i = \Aut_\CF^{n}(V_i),
\end{gather*}
for $0\leq i\leq p$.  Note that $\Aut_\CF(V_i)$ is closed under left and right $\Aut_S(V_i)$-action in $\Aut(V_i)$.  Thus $\Aut_\CF(V_i)\cap T_i$ is a set of representatives of extendable $\CF$-automorphisms of $V_i$, and $\Aut_\CF(V_i)\cap R_i$ is a set of representatives of nonextendable $\CF$-automorphisms of $V_i$, both 
with respect to $S$-$S$-conjugacy.  

The subgroup $V_i$ is $\CF$-radical if and only if $\Aut_\CF(V_i)$ contains $\SL_2(p)$.  In this case, we 
have 
\[
	\Aut_\CF(V_i)\cong \SL_2(p):r_i \quad\text{ for some } r_i|(p-1).
\]
Then
\begin{gather*}
	\Aut_\CF(V_i)\cap T_i = \left\{ \begin{pmatrix} k & 0\\ 0 & l \end{pmatrix}_{V_i} \mid k,l\in
\mathbb{F}_p^{\times}, kl\in\la \mu^{(p-1)/r_i}\ra \right\},\\
	\Aut_\CF(V_i)\cap R_i = \left\{ \begin{pmatrix} 0 & l\\ k & 0 \end{pmatrix}_{V_i} \mid k,l\in
\mathbb{F}_p^{\times}, -kl\in\la \mu^{(p-1)/r_i}\ra \right\},
\end{gather*}
where $\mathbb{F}_p^{\times}=\la\mu\ra$.  In particular, we have 
\[	
	|\Aut_\CF(V_i)\cap T_i|=|\Aut_\CF(V_i)\cap R_i|=(p-1)r_i.
\]
For later use, set $\Lambda^e_i=\{ (k,l)\mid kl\in\la \mu^{(p-1)/r_i}\ra \}$, 
$\Lambda^{n}_i=\{ (k,l)\mid -kl\in\la \mu^{(p-1)/r_i}\ra \}$.

We will consider saturated fusion systems $\CF$ on $S$ such that all $V_i$ are $\CF$-radical.  In particular, all exotic fusion systems on $S$ have this property.  We list all of them in the following table.  
\begin{center}
\renewcommand{\arraystretch}{1.5}
\begin{tabular}{|c|c|c|c|c|}
\hline
$p$ & $\Out_\CF(S)$ & $|V_i^{\CF}|$ & $r_i$ & Group\\
\hline
3 & $D_8$ & $2, 2$ & $2, 2$ & ${}^{2}F_4(2)'$\\\hline
3 & $SD_{16}$ & 4 & $2$ & $J_4$\\\hline
5 & $4S_4$ & 6 & $4$ & $Th$\\\hline
7 & $D_{16}\times 3$ & $4, 4$ & $2, 2$ & \\\hline
7 & $6^2:2$ & $6, 2$ & $2$, $6$ &  \\\hline
7 & $SD_{32}\times 3$ & $8$ & $2$ & \\\hline
\end{tabular}
\end{center}
Here $|V_i^\CF|$ denotes the size of the $\CF$-conjugacy class of subgroups of $S$ containing $V_i$. In the columns $|V_i^\CF|$ and $r_i$, numbers corresponding to different $\CF$-conjugacy classes are separated by commas, and they are in the same order in these two columns.  The column ``Group'' indicates finite groups realizing corresponding saturated fusion systems.  An empty entry in this column means that the corresponding saturated fusion system is exotic.  For example, when $p=7$, the saturated fusion system on $S$ with $\Out_\CF(S) \cong 6^2:2$ has exactly two $\CF$-conjugacy classes of subgroups of order $7^2$.  One of them consists of 6 subgroups of order $7^2$ and their automorphism groups are isomorphic to $\SL_2(7):2$; the other consists of 2 subgroups of order $7^2$ and their automorphism groups are isomorphic to $\SL_2(7):6 = \GL_2(7)$.  This saturated fusion system is exotic.

One observes that the numerical relation $|\Out_\CF(S)|=(p-1)|V_i^\CF|r_i$ is satisfied in all cases.  This is not a coincidence, and indeed a special case of a more general phenomenon.

\begin{lem} \label{T:f-number}
Let $\CF$ be a saturated fusion system on $S$. Suppose that $V_i$ is $\CF$-radical for some $1\leq i\leq p$.  Then we have
\[
	f:=\frac{|\Out_\CF(S)|}{p-1}=|V_i^\CF|r_i=|\{ [\alpha]\in\Out_\CF(S)\mid \alpha(z)=z^m \}|
\]
for any $m\in\BF_p^\times$.
\end{lem}

\begin{proof}
By Proposition~\ref{restriction isom}, we have
\[
	|\Out_\CF(S)|=|\Hom_\CF^e(V_i,S)/\sim_{(S,S)}|.
\]
We claim that
\[
	|\Hom_\CF^e(V_i,S)/\sim_{(S,S)}| = |V_i^\CF| \cdot |\Aut_\CF^e(V_i)/\sim_{(S,S)}|.
\]
Indeed, if $V_i \cong_\CF V_j$, then there is $\alpha\in\Aut_\CF(S)$ such that $\alpha(V_i)=V_j$ by Alperin's fusion theorem.  Then composition with $\alpha$ induces a bijection
\[
	\Aut^e_\CF(V_i) /\sim_{(S,S)} \xrightarrow{\simeq} \{ \varphi \in \Hom^e_\CF(V_i,S) \mid \varphi(V_i) = V_j \} /\sim_{(S,S)}.
\]
Since this holds for every $j$ with $V_i\cong_\CF V_j$, the claim follows.  Since 
\[
	|\Aut_\CF^e(V_i)/\sim_{(S,S)}| = |\Aut_\CF(V_i) \cap T_i| = (p-1)r_i,
\]
as we have observed previously in this section, the first equaltiy of the proposition follows.

Now consider the restriction map $\Aut_\CF(S) \to \Aut_\CF(Z(S))$.  Since inner automorphisms of $S$ acts as the identity map on $Z(S)$, this restriction map induces the group homomorphism
\[
	\Out_\CF(S) \to \Aut_\CF(Z(S)),\quad [\alpha] \mapsto \alpha|_{Z(S)}.
\]
For each $m\in\BF_p^\times$, we have $\begin{pmatrix} m & 0 \\ 0 & m^{-1} \end{pmatrix}_{V_i} \in \Aut_\CF^e(V_i)$ as $ \Aut_\CF^e(V_i)$ contains $\SL_2(p)$, and hence there exists $\alpha\in\Aut_\CF(S)$ such that $\alpha(z)=z^m$.  For each pair $m,n\in\BF_p^\times$, choose $\beta\in\Aut_\CF(S)$ such that $\beta(z)=z^{nm^{-1}}$.  Then composition with $\beta$ defines the bijection
\[
	\{ [\alpha]\in\Out_\CF(S)\mid \alpha(z)=z^m \} \to \{ [\alpha]\in\Out_\CF(S)\mid \alpha(z)=z^n \}.
\]
Thus we get the second equality of the proposition.
\end{proof}

Now suppose that $\CF$ is a saturated fusion system on $S$ such that all $V_i$ are $\CF$-radical.  Let us enumerate $\CF$-morphisms between subgroups of $S$ of index $p$ up to $S$-$S$-conjugacy.  Suppose $V_i \cong_\CF V_j$.  By Alperin's fusion theorem, there is $\alpha\in\Aut_\CF(S)$ such that $\alpha(V_i)=V_j$.  Then $\alpha(Z(S))=Z(S)$ and $\alpha(u_i)\in u_j^k\la z\ra$ for some $0 < k < p$.  We normalize the notations for the sake of convenience as follows.  Since $\Aut_\CF(V_j)\geq \SL_2(p)$, by composing $\alpha$ with some extendable $\CF$-automorphism of $V_j$ and replacing $u_j$ by a suitable power of $u_j$, we obtain $\alpha_{i,j}\in\Aut_\CF(S)$ sending $z$ to $z$ and $u_i$ to $u_j$.  We do this one by one for each $\CF$-conjugacy class of the $V_i$ in a compatible way.  By composing elements of $\Aut_\CF(V_i)\cap R_i$ with $\alpha_{i,j}$, we obtain nonextendable $\CF$-isomorphisms
\[
	\varphi_{i,j}^{k,l}\colon V_i \to V_j
\]
sending $z$ to $u_j^k$ and $u_i$ to $z^l$ where $(k,l)\in\Lambda^n_i$.  The set
\[
	\{ \varphi_{i,j}^{k,l} \mid V_i\cong_\CF V_j, (k,l)\in\Lambda^n_i \}
\]
is a set of representatives of nonextendable isomorphisms between index $p$ subgroups of $S$ with 
respect to $S$-$S$-conjugacy.  Similarly we have a set
\[
	\{ \psi_{i,j}^{k,l} \mid V_i\cong_\CF V_j, (k,l)\in\Lambda^e_i \}
\]
of representatives of extendable isomorphisms between index $p$ subgroups of $S$ with respect to $S$-$S$-conjugacy, where
\[
	\psi_{i,j}^{k,l}\colon V_i \to V_j
\]
sends $z$ to $z^k$ and $u_i$ to $u_j^l$.

Then, by Proposition~\ref{T:layer1}, if $X$ is a right characteristic biset for $\CF$ we have
\begin{gather}
	X_0 \cong c_0 \coprod_{[\alpha]\in\Out_\CF(S)} (S\times S)/\Delta^{\alpha}_{S}, \tag{$X_0$} \label{E:X0 with undetermined coeffs}\\
	X_1 \cong \coprod_{\substack{V_i\cong_\CF V_j\\ (k,l)\in\Lambda^e_i}} c_1^{(i)}(S\times S)/
\Delta^{\psi_{i,j}^{k,l}}_{V_i} \amalg \coprod_{\substack{V_i\cong_\CF V_j\\ (k,l)\in\Lambda^{n}_i}} 
(c_0+pc_1^{(i)})(S\times S)/\Delta^{\varphi_{i,j}^{k,l}}_{V_i}, \tag{$X_1$} \label{E:X1 with undetermined coeffs}
\end{gather}
for some nonnegative integers $c_0$ and $c_1^{(i)}$ such that $p \nmid c_0$.

Now we consider $\CF$-maps between subgroups of order $p$.  Since we are assuming that all $V_i$ are $\CF$-radical, all elements of $S$ of order $p$ are $\CF$-conjugate.  For elements $\xi, \zeta\in S$ of order $p$, let $\Delta_{\xi}^{\zeta}:=\Delta_{\la \xi \ra}^{\varphi}$ where $\varphi\colon \la\xi\ra \to \la\zeta\ra$ sends $\xi$ to $\zeta$.  Then the $\CF$-graph subgroups of $S\times S$ (i.e. subgroups of $S\times S$ of the form $\Delta_Q^\varphi$ for some $Q\leq S$ and $\varphi\in\Hom_\CF(Q,S)$) of order $p$ up to $S\times S$-conjugacy are
\[
	\Delta_{\xi}^{\zeta}\quad\text{where } \xi\in\{ z, u_i \}_{0\leq i\leq p}, \zeta\in\{ z^m, u_j^m \}_{0\leq j\leq 
p, 1\leq m\leq p-1}.
\]
Thus
\begin{align} \label{E:X2 with undetermined coeffs}
	X_2 \cong \coprod_{1\leq m\leq p-1}c_2^{(z,z^m)}(S\times S)/\Delta_{z}^{z^m} 
	\amalg \coprod_{\substack{0\leq i\leq p\\ 1\leq m\leq p-1}}c_2^{(u_i,z^m)}(S\times S)/\Delta_{u_i}^{z^m} \tag{$X_2$}\\ 
	\amalg \coprod_{\substack{0\leq j\leq p\\ 1\leq m\leq p-1}}c_2^{(z,u_j^m)}(S\times S)/
\Delta_{z^m}^{u_j^m} 
	\amalg \coprod_{\substack{0\leq i, j\leq p\\ 1\leq m\leq p-1}}c_2^{(u_i,u_j^m)}(S\times S)/
\Delta_{u_i}^{u_j^m} \notag
\end{align}
for some nonnegative integers $c_2^{(z,z^m)}$, $c_2^{(u_i,z^m)}$, $c_2^{(z,u_j^m)}$ and $c_2^{(u_i,u_j^m)}$.

Now we can state the main result of this section precisely.

\begin{thm} \label{T:minimal biset for RV}
Let $p$ be an odd prime and let $S$ be an extraspecial group of order $p^3$ and exponent $p$.  Let $\CF$ be a saturated fusion system on $S$ such that all subgroups $V_i$ of $S$ of order $p^2$ are $\CF$-radical.  Then there exists a unique minimal right characteristic biset X for $\CF$ given as follows:
\begin{gather*}
	X_0 \cong \coprod_{[\alpha]\in\Out_\CF(S)} (S\times S)/\Delta^{\alpha}_{S},\\
	X_1 \cong \coprod_{\substack{V_i\cong_\CF V_j\\ (k,l)\in\Lambda^{n}_i}} (S\times S)/
\Delta^{\varphi_{i,j}^{k,l}}_{V_i},\\
	X_2 \cong \coprod_{\substack{V_i\cong_\CF V_j \\  1\leq m\leq p-1}} (f-r_i)(S\times S)/\Delta_{u_i}^{u_j^m} 
\amalg \coprod_{\substack{V_i\not\cong_\CF V_j \\  1\leq m\leq p-1}} f(S\times S)/\Delta_{u_i}^{u_j^m}.
\end{gather*}
Moreover, $X$ is a unique minimal left characteristic biset for $\CF$ and hence a unique minimal characteristic biset for $\CF$.
\end{thm}

\begin{proof}
The proof is done by explicitly computing all right characteristic bisets for $\CF$ and showing that there is the smallest one among them.  For this, let $X$ be an arbitrary right characteristic biset for $\CF$.  The top two layers $X_0$ and $X_1$ of $X$ are given as in \eqref{E:X0 with undetermined coeffs} and \eqref{E:X1 with undetermined coeffs} by Proposition~\ref{T:layer1}.  Write $X_2$ as in \eqref{E:X2 with undetermined coeffs}.  The coefficients of $X_2$ are completely determined by the system of equations
\[
	|X^{\Delta_{\xi}^{\zeta}}|=|X^{\Delta_{\xi}^{\xi}}|,\qquad \xi\in\{ z, u_i \}_{0\leq i\leq p}, \zeta\in\{ z^m, u_j^m \}_{0\leq j\leq 
p, 1\leq m\leq p-1}.
\]
To solve this system of equations, first compute the numbers of fixed points of $\Delta_{\xi}^{\zeta}$ on transitive subbisets of $X$ using Lemma~\ref{T:no of fixed pts} and Lemma~\ref{T:no of fixed pts on top piece}:
\begin{align*}
	|((S\times S)/\Delta_{S}^{\alpha})^{\Delta_{\xi}^{\zeta}}|&=
	\begin{cases}
		p^3,	&\text{if $\alpha(\xi)=\zeta\in\la z\ra$}\\
		p^2,	&\text{if $\alpha(\xi)\in\zeta\la z\ra \neq \la z \ra$}\\
		0,	&\text{otherwise},
	\end{cases}\\
	|((S\times S)/\Delta_{V_i}^{\psi_{i,j}^{k,l}})^{\Delta_{\xi}^{\zeta}}|&=
	\begin{cases}
		p^4,	&\text{if $(\xi, \zeta)=(z,z^k)$}\\
		p^3,	&\text{if $(\xi, \zeta)=(u_i,u_{j}^{l})$}\\
		0,	&\text{otherwise},
	\end{cases}\\
	|((S\times S)/\Delta_{V_i}^{\varphi_{i,j}^{k,l}})^{\Delta_{\xi}^{\zeta}}|&=
	\begin{cases}
		p^3,	&\text{if $(\xi,\zeta)=(z, u_j^k)$ or $(\xi, \zeta ) = (u_i,z^l)$}\\
		p^2,	&\text{if $\xi=u_i$, $\zeta\in\la u_j\ra$}\\
		0,	&\text{otherwise},
	\end{cases}\\
	|((S\times S)/\Delta_{\xi}^{\zeta})^{\Delta_{\xi}^{\zeta}}|&=
	\begin{cases}
		p^5,	&\text{if $\xi, \zeta\in\la z \ra$}\\
		p^4,	&\text{if $\xi\in\la z\ra\not\ni\zeta$ or $\xi\notin\la z\ra\ni\zeta$}\\
		p^3,	&\text{if $\xi\notin\la z\ra\not\ni\zeta$.}	
	\end{cases}
\end{align*}

Adding up those numbers, we get the following result using Lemma~\ref{T:f-number}:
\begin{align*}
	|X_0^{\Delta_{z}^{z^m}}|&=p^3fc_0,\qquad |X_0^{\Delta_{u_i}^{z^m}}|=|X_0^{\Delta_{z}^{u_j^m}}|=0,\\
	|X_0^{\Delta_{u_i}^{u_j^m}}|
	&=c_0p^2|\{ [\alpha]\in\Out_\CF(S)\mid \alpha(u_i)\in u_j^m\la z\ra \}|\\
	&=
	\begin{cases}
		p^2r_ic_0,	&\text{if $V_i\cong_\CF V_j$}\\
		0,			&\text{otherwise},
	\end{cases},\\
	|X_1^{\Delta_{z}^{z^m}}|&=p^4f\sum_{0\leq i\leq p}c_1^{(i)},\qquad |X_1^{\Delta_{u_i}^{z^m}}|=p^3fc_0+p^4fc_1^{(i)}\\	
	|X_1^{\Delta_{z}^{u_j^m}}|&=p^3fc_0+p^4r_j\sum_{V_i\cong_\CF V_j}c_1^{(i)},\\
	|X_1^{\Delta_{u_i}^{u_j^m}}|&=
	\begin{cases}
		p^3r_ic_1^{(i)}+p^2(p-1)r_i(c_0+pc_1^{(i)}),	&\text{if $V_i\cong_\CF V_j$}\\
		0,				&\text{otherwise}.
	\end{cases}
\end{align*}
The second equality for $|X_0^{\Delta_{u_i}^{u_j^m}}|$ needs more explanation.  If there is $\alpha\in\Aut_\CF(S)$ such that $\alpha(u_i)\in u_j^m\la z\ra$, then $\alpha(V_i)=V_j$.  Thus if $V_i\not\cong_\CF V_j$, then $\{ \alpha\in\Out_\CF(S)\mid \alpha(u_i)\in u_j^m\la z\ra \}=\emptyset$.  Suppose $V_i\cong_\CF V_j$.  Then by the extension axiom, there is $\alpha\in\Aut_\CF(S)$ such that $\alpha(V_i)=V_j$.  Then $ \alpha(u_i)\in u_j^l\la z\ra$ for some $l\neq 0$.  Since $\Aut_\CF(V_i)\geq \SL_2(p)$, we may modify $\alpha$ so that $\alpha(u_i)=u_j^m$.  Then $\alpha(z)=z^k$ for some $k$ and there are $r_i$ choices for $k$.  By~\cite[A8]{BrotoLeviOliver2003b}, each choice of $k$ determines $\alpha$ up to $S$-$S$-conjugacy.  

Thus
\begin{align*}
	|X_{\leq 1}^{\Delta_z^{z^m}}|&=p^3fc_0+p^4f\sum_{0\leq i\leq p}c_1^{(i)},\\
	|X_{\leq 1}^{\Delta_{u_i}^{z^m}}|&=p^3fc_0+p^4fc_1^{(i)},\\
	|X_{\leq 1}^{\Delta_{z}^{u_j^m}}|&=p^3fc_0+p^4r_j\sum_{V_i\cong_\CF V_j}c_1^{(i)},\\
	|X_{\leq 1}^{\Delta_{u_i}^{u_j^m}}|&=
	\begin{cases}
		p^3r_ic_0+p^4r_ic_1^{(j)},	&\text{if $V_i\cong_\CF V_j$}\\
		0,				&\text{otherwise}.
	\end{cases}
\end{align*}

Now $|X^{\Delta_{\xi}^{\zeta}}|=|X^{\Delta_{\xi}^{\xi}}|$ and $|X^{\Delta_{\xi}^{\zeta}}|=|X_{\leq 
1}^{\Delta_{\xi}^{\zeta}}|+c_2^{(\xi,\zeta)}|((S\times S)/\Delta_{\xi}^{\zeta})^{\Delta_{\xi}^{\zeta}}|$. Thus
\[
	c_2^{(\xi,\zeta)}|((S\times S)/\Delta_{\xi}^{\zeta})^{\Delta_{\xi}^{\zeta}}|=c_2^{(\xi,\xi)}|((S\times S)/
\Delta_{\xi}^{\xi})^{\Delta_{\xi}^{\xi}}|+(|X_{\leq 1}^{\Delta_{\xi}^{\xi}}|-|X_{\leq 
1}^{\Delta_{\xi}^{\zeta}}|).
\]
Then
\begin{align*}
	|X^{\Delta_z^{z^m}}|=|X^{\Delta_z^{z}}|&: c_2^{(z,z^m)}=c_2^{(z,z)}=:c_2^{(z)},\\
	|X^{\Delta_{u_i}^{u_j^m}}|=|X^{\Delta_{u_i}^{u_i}}|&: c_2^{(u_i,u_j^m)}=c_2^{(u_i,u_i)}=:c_2^{(u_i)} 
(\text{if } V_i\cong_\CF V_j),\\
	|X^{\Delta_{u_i}^{u_j^m}}|=|X^{\Delta_{u_i}^{u_i}}|&: c_2^{(u_i,u_j^m)}=r_ic_0+pr_ic_1^{(i)} + c_2^{(u_i)} (\text{if } V_i\not\cong_\CF V_j),\\
	|X^{\Delta_{z}^{u_j^m}}|=|X^{\Delta_{z}^{z}}|&: c_2^{(z,u_j^m)}=pc_2^{(z)}+(f-r_j)\sum_{V_i\cong_\CF V_j}c_1^{(i)}+f\sum_{V_i\not\cong_\CF V_j}c_1^{(i)},\\
	|X^{\Delta_{u_i}^{z^m}}|=|X^{\Delta_{u_i}^{u_i}}|&: pc_2^{(u_i,z^m)}=c_2^{(u_i)}-(f-r_i)c_0-p(f-
r_i)c_1^{(i)}.
\end{align*}
The only restriction on coefficients (other than $p \nmid c_0$) is that they must be nonnegative because  $X$ is a genuine biset:
\[
	p \nmid c_0\geq 1,\quad c_1^{(i)}\geq 0,\quad c_2^{(z)}\geq 0,\quad c_2^{(u_i)}\geq (f-r_i)c_0+p(f-
r_i)c_1^{(i)}.
\]
Thus $X$ is minimal if and only if
\[
	c_0=1,\quad c_1^{(i)}=0,\quad c_2^{(z)}=0,\quad c_2^{(u_i)}=(f-r_i),
\]
and in this case 
\begin{align*}
	c_2^{(u_i,u_j^m)} &= 
	\begin{cases}
		f-r_i,	&\text{if $V_i \cong_\CF V_j$}\\
		f,	&\text{if $V_i \not\cong_\CF V_j$},\\
	\end{cases}\\
	c_2^{(u_i,z^m)} & = c_2^{(z,u_j^m)} = 0.
\end{align*}
The last statement follows from observing that the unique minimal right characteristic biset $X$ is isomorphic to its opposite $X^\op$.
\end{proof}

\begin{rmk} \label{R:numerical rel}
Let us look more closely at numerical relations in the above minimal characteristic biset $X$.  Let $e:=e(X)=|X|/|S|$ and $e_i:=|X_i|/|S|$ $(i=0,1,2)$ so that $e=e_0+e_1+e_2$.  Let  $d_i=e_i/p^i$ $(i=0,1,2)$.  In other words, $d_i$ is the number of transitive subbisets appearing in $X_i$.  Then, using Lemma~\ref{T:f-number}, we see that
\begin{align*}
	d_0 &= |\Out_\CF(S)|,\\
	d_1 &= \sum_{0\leq i\leq p} |V_i^\CF||\Lambda_i^n| = \sum_{0\leq i\leq p} |V_i^\CF|(p-1)r_i = (p+1)|\Out_\CF(S)|,\\
	d_2 &= \sum_{0\leq i\leq p} (f-r_i)|V_i^\CF|(p-1)+\sum_{0\leq i\leq p}f((p+1)-|V_i^\CF|)(p-1)=p(p+1)|\Out_\CF(S)|.
\end{align*}
Thus
\[
	e=d_0+pd_1+p^2d_2=\frac{p^5-1}{p-1}|\Out_\CF(S)|.
\]

We summarize all these numbers in the following table.

\begin{center}
\renewcommand{\arraystretch}{1.5}
\begin{tabular}{|c|c|c|c|c|c|c|c|c|c|c|}
\hline
$p$ & $\Out_\CF(S)$ & $|V_i^{\CF}|$ & $r_i$ & $f$ & $d_0$ & $d_1$ & $d_2$ 
& $e(X)$ & Group\\
\hline
3 & $D_8$ & $2, 2$ & $2, 2$ & $4$ & $8$ & $32$ & $96$ & $968$ & ${}^{2}F_4(2)'$\\\hline
3 & $SD_{16}$ & 4 & $2$ & $8$ & 16 & $64$ & $192$ &$1936$ & $J_4$\\\hline
5 & $4S_4$ &6 & $4$ & $24$ & 96 & $576$ & $2880$ & $74976$ & $Th$\\\hline
7 & $D_{16}\times 3$ & $4, 4$ & $2, 2$ & $8$ & $48$ & $384$ & $2688$ & $134448$ & $\leq 425744$ \\\hline
7 & $6^2:2$ & $6, 2$ & $2$, $6$ & $12$ & $72$ & $576$ & $4032$ & $201672$ & $\leq 638620$ \\\hline
7 & $SD_{32}\times 3$ & $8$ & $2$ & $16$ & $96$ & $768$ & $5376$ & $268896$ & $\leq 851496$\\\hline
\end{tabular}
\end{center}
Numbers in the column ``Group'' are the upper bounds of the exoticity indices obtained by minimal characteristic bisets.
\end{rmk}

\section{Finite groups realizing Ruiz-Viruel exotic fusion systems} \label{S:finite group for RV}

In this section, we analyze how the finite group $G$ realizes fusion more closely, and as a consequence prove Theorem~\ref{T:main2}.  For this, the following basic double coset formula concerning restriction of bisets will be useful.

\begin{lem}[{\cite[\S 2]{Ragnarsson2006}}] \label{T:res-biset}
Let $S$ be a group.  Let $Q,R\leq S$ and let $\varphi\colon Q\to S$, $\psi\colon R\to S$ be injective group homomorphisms.  Then
\[
	{}_{\psi}((S\times S)/\Delta_{Q}^{\varphi}) \cong \coprod_{t\in [\varphi(R)\backslash S/Q]} (R\times S)/\Delta_{\psi^{-1}({}^{t}Q)}^{\varphi\circ c_{t}^{-1}\circ \psi}
\]
\end{lem}

Let $\CF$ be a saturated fusion system on a finite $p$-group $S$ with a left characteristic biset $X$.  Let $G=\Aut({}_{1}X)$, i.e.\ the group of bijections of $X$ preserving the right $S$-action, and define $\iota\colon S\to G$ by $\iota(u)(x)=ux$ for $u\in S$, $x\in X$.  By our previous work~\cite{Park2010Realizing}, we know that $\iota$ is injective and $\CF\cong\CF_{\iota(S)}(G)$.  More precisely, let $R\leq S$ and let $\psi\in\Hom_\CF(R,S)$.  Then the left $\CF$-stability of $X$ implies that ${}_{R}X$ and ${}_{\psi}X$ are isomorphic as $R$-$S$-bisets.  Then any isomorphism of $R$-$S$-bisets $g \colon {}_{R}X \xrightarrow{\sim} {}_{\psi}X$ viewed as an element of $G$ {\em realizes} $\psi$ in the sense that $\psi(u)=gug^{-1}$ for all $u\in R$.  

In fact, we can be more precise about the element $g\in G$. Let us recall the notations in \cite{Park2010Realizing}.  Write 
\[
	X=\coprod_{i=1}^{n} S\times_{(Q_i,\varphi_i)} S
\]
where $Q_i\leq S$, $\varphi_i \in \Hom_\CF(Q_i,S)$ for $1\leq i\leq n$.  For each $i$, fix a set $\{t_{ij}\}_{j \in J_i}$ of 
representatives of the left cosets of $Q_i$ in $S$.  Set $J=\coprod_{i=1}^{n}J_i$.  For each $1\leq i\leq n$, we have a decomposition of right $S$-sets
\[
	S \times_{(Q_i,\varphi_i)} S = \coprod_{j \in J_i} \la t_{ij}, S\ra,
\]
where $\la t_{ij},S\ra:=\{ \la t_{ij},x\ra \mid x\in S \}$ is a regular right $S$-set for all $j\in J_i$. So we have
\[
	G:=\Aut({}_{1}X) \cong S\wr\Sym(J),
\]
where $\Sym(J)$ denotes the symmetric group on the set $J$.  

Let $\ol{G}$ be the image of $G$ in $\Sym(J)$ and use bar notation for images of subgroups or elements of $G$ in $\Sym(J)$.  Then each $J_i$ is an $\ol{S}$-orbit with $|J_i|=|S:Q_i|$.  In particular, we have $|J_i|=1$ precisely when the corresponding $\CF$-morphism $\varphi_i\colon Q_i\to S$ is an automorphism of $S$.  Suppose that the $Q_i$ are indexed so that $|Q_i|\geq |Q_{i+1}|$ for all $i$.  Then there are integers $1\leq n_0\leq n_1\leq n_2\leq\cdots\leq n$ such that $|S:Q_i|=p^{r}$ for all $i$ with $n_r \leq i < n_{r+1}$. With this notation, $\ol{S}$-orbits of $J$ can be represented schematically as follows:
\begin{align*}
	\overbrace{\ast\ast\ast\ast\ast\ast\ast\,\ast}^{[\varphi_{i}](n_0\leq i< n_1)}
\overbrace{(\ast\ast\ast)\cdots(\ast\ast\ast)}^{[\varphi_{i}](n_1\leq i< n_2)}
\overbrace{(\ast\ast\ast\ast\ast\ast\ast\ast\ast)\cdots(\ast\ast\ast\ast\ast\ast\ast\ast\ast)}^{[\varphi_{i}](n_2\leq i< n_3)} \cdots
\end{align*}
Each asterisk in the above diagram represents an element of the index set $J$.  Left $\ol{S}$-orbits of the asterisks are denoted by brackets, except for the first eight asterisks, each of which is a single left $\ol{S}$-orbit.  Each left $\ol{S}$-orbit corresponds to a transitive subbiset of $X$ of the form $S\times_{(Q_i,\varphi_i)} S$, and indexed by the $S$-$S$-conjugacy class of $\CF$-morphisms $[\varphi_i]$ determining it.  In a left $\ol{S}$-orbit indexed by the $S$-$S$-conjugacy class of $\varphi_i\colon Q_i\to S$, each asterisk is indexed by a left coset of $Q_i$ in $S$.  

\begin{rmk} \label{R:fusion-corr-rule}
With the above notation in mind, suppose $R\leq S$, $\psi\in\Hom_\CF(R,S)$.  Then ${}_{R}X \cong {}_{\psi}X$, and by Lemma~\ref{T:res-biset},
\begin{align*}
	{}_{\psi}((S\times S)/\Delta_{Q_i}^{\varphi_i}) &\cong \coprod_{t\in [\psi(R)\backslash S/Q_i]} (R\times S)/\Delta_{\psi^{-1}({}^{t}Q_i)}^{\varphi_i\circ c_{t}^{-1}\circ \psi},\\
	{}_{R}((S\times S)/\Delta_{Q_i}^{\varphi_i}) &\cong \coprod_{s\in [R\backslash S/Q_i]} (R\times S)/\Delta_{{}^{s}Q_i}^{\varphi_i\circ c_{s}^{-1}}.
\end{align*}
So one can find a bijective correspondence $\mu$ between the above transitive subbisets of ${}_{R}X$ and ${}_{\psi}X$ and a permutation $\wt{\mu}\in\Sym(J)$ which respects $\mu$, and construct an isomorphism of $R$-$S$-bisets $g \colon {}_{R}X \xrightarrow{\sim} {}_{\psi}X$ which, viewed as an element of $G$, realizes $\psi$ and such that $\ol{g}=\wt{\mu}$.
\end{rmk}

Our focus is the behavior of the permutations $\ol{g}\in\Sym(J)$ associated with elements $g\in G$ realizing $\CF$-automorphisms of the $\CF$-essential subgroups of $S$.  First we consider $\CF$-automorphisms of $S$.

\begin{prp} \label{T:fusion permutations for Aut(S)}
Let $\CF$ be a saturated fusion system on a finite $p$-group $S$ with a left characteristic biset $X$.  Let $G=\Aut({}_{1}X)$ and use the notations introduced above.  Then the image of $\Aut_\CF(S)$ in $\Sym(J)$ is transitive on $J^{(0)} := \coprod \{ J_i \mid |J_i|=1 \}$.
\end{prp}

\begin{proof}
$J^{(0)}$ is the union of singleton left $\ol{S}$-orbits of $J$, each corresponding to an outer automorphism of $S$.  By Remark~\ref{R:fusion-corr-rule}, for any $\alpha\in\Aut_\CF(S)$ there is $g\in G$ which realizes $\alpha$ and such that $\ol{g}\in\Sym(J)$ sends the element of $J^{(0)}$ corresponding to $[\id_S]$ to the element of $J^{(0)}$ corresponding to $[\alpha]$ because 
\[
	{}_{\alpha}((S\times S)/\Delta_S^{\varphi_i}) \cong (S\times S)/\Delta_S^{\varphi_i\circ\alpha}
\]
for $Q_i=S$.
\end{proof}

Now we can prove Theorem~\ref{T:main2}, which we restate here for the convenience of the reader.

\begin{thm} \label{T:main2'}
Let $p$ be an odd prime and let $S$ be an extraspecial group of order $p^3$ and exponent $p$.  Let $\CF$ be a saturated fusion system on $S$ such that all subgroups of $S$ of order $p^2$ are $\CF$-radical.  Let $X$ be the minimal left characteristic biset $X$ for $\CF$ given by Theorem~\ref{T:main1}.  Let $G = \Aut({}_{1}X) \cong S \wr \Sigma_{e(X)}$ and let $H$ be the subgroup of $G$ generated by the normalizers of the $\CF$-essential subgroups of $S$.  Then the image of $H$ in $\Sigma_{e(X)}$ is a transitive subgroup of $\Sigma_{e(X)}$.
\end{thm}

\begin{proof}
We keep the notations of this section and \S\ref{S:minimal biset for RV}.  We have
\[
	H = \la \Aut_\CF(S), \Aut_\CF(V_i) \mid 1\leq i\leq n \ra.
\]
The index set $J$ of the symmetric group $\Sigma_{e(X)}$ can be represented schematically as follows: 
\begin{align*}
	\overbrace{\ast\ast\ast\ast\ast\ast\ast\,\ast}^{[\alpha]}
\overbrace{(\ast\ast\ast)\cdots(\ast\ast\ast)}^{[\varphi_{0,j}^{k,l}]} \cdots \overbrace{(\ast\ast\ast)\cdots(\ast\ast\ast)}^{[\varphi_{p,j}^{k,l}]}
\overbrace{(\ast\ast\ast\ast\ast\ast\ast\ast\ast)\cdots(\ast\ast\ast\ast\ast\ast\ast\ast\ast)}^{[u_0\to u_j^m]}\cdots
\end{align*}
By Proposition~\ref{T:fusion permutations for Aut(S)}, we know that the image of $\Aut_\CF(S)$ in $\Sym(J)$ is transitive on the union of singleton left $\ol{S}$-orbits of $J$ corresponding to outer automorphisms of $S$. 

Now let $\varphi=\varphi_{0,0}^{r,s}\in\Aut_\CF(V_0)$ for some $(r,s)\in\Lambda_0^n$.  Recall that $Z(S)=\la z \ra$, $V_i=\la z,u_i \ra$ for $0\leq i\leq p$. Let $v\in S$ such that $[v,u_0]=z$ so that $v^tu_0v^{-t} = z^tu_0$ for all $t$.  Then we have sets of left coset representatives as follows:
\begin{gather*}
	[S/V_0]=\{v^t\}_{0\leq t\leq p-1},\quad [S/V_i]=\{u_0^t\}_{0\leq t\leq p-1} (i\neq 0),\\
	[S/\la u_0\ra]=\{ v^tz^w \}_{0\leq t,w\leq p-1},\quad [S/\la u_i\ra]=\{ u_0^tz^w\}_{0\leq t,w\leq p-1} (i\neq 0).
\end{gather*}
We have ${}_{V_0}X\cong{}_{\varphi}X$ as $V_0$-$S$-bisets, and for each transitive subbiset of $X$, the two restrictions to $V_0$ via the inclusion $V_0 \inj S$ and via $\varphi$ are given as follows:
\begin{center}
\renewcommand{\arraystretch}{1.5}
\begin{tabular}{|c|c|c|}
\hline
$Y$ & ${}_{V_0}Y$ & ${}_{\varphi}Y$ \\
\hline
$(S\times S)/\Delta^{\alpha}_{S}$ & 
$(V_0\times S)/\Delta_{V_0}^{\alpha|_{V_0}}$ & 
$(V_0\times S)/\Delta_{V_0}^{\alpha\circ\varphi}$ \\
\hline
$(S\times S)/\Delta^{\varphi_{0,j}^{k,l}}_{V_0}$ & 
$\coprod_{t=0}^{p-1} (V_0\times S)/\Delta_{V_0}^{\varphi_{0,j}^{k,l}\circ c_{v^{t}}}$ &
$\coprod_{t=0}^{p-1} (V_0\times S)/\Delta_{V_0}^{\varphi_{0,j}^{k,l}\circ c_{v^{t}}\circ\varphi}$ \\
\hline
$(S\times S)/\Delta^{\varphi_{i,j}^{k,l}}_{V_i} (i\neq 0)$ &
$(V_0\times S)/\Delta_{\la z \ra}^{u_j^{rk}}$ & 
$(V_0\times S)/\Delta_{\la u_0 \ra}^{u_j^{rl}}$ \\
\hline
$(S\times S)/\Delta^{u_j^m}_{\la u_0 \ra}$ & 
$\coprod_{t=0}^{p-1} (V_0\times S)/\Delta_{\la z^{t}u_0\ra}^{u_j^m}$ & 
$\coprod_{t=0}^{p-1} (V_0\times S)/\Delta_{\la z^{r^{-1}}u_0^{s^{-1}t}\ra}^{u_j^m}$ \\
\hline
$(S\times S)/\Delta^{u_j^m}_{\la u_i \ra} (i\neq 0)$ & 
$V_0\times S$ & 
$V_0\times S$\\
\hline
\end{tabular}
\end{center}
By comparing the size, we see that the transitive $V_0$-$S$-subbisets in the first two rows of the above table are isomorphic to each other, and those in the third and fourth rows are isomorphic to each other.  Among the transitive subbisets in the first two rows, we can determine isomorphic pairs by the extendability of the associated maps.  In the column ${}_{V_0}Y$, the maps $\alpha|_{V_0}$ are extendable, while $\varphi_{0,j}^{k,l}\circ c_{v^{t}}$ are nonextendable.  On the other hand, in the column ${}_{\varphi}Y$, the maps $\alpha\circ\varphi$ are nonextendable, while the maps $\varphi_{0,j}^{k,l}\circ c_{v^{t}}\circ\varphi$ are extendable if $t=0$ and nonextentable if $t\neq 0$.  Thus
\begin{align*}
	(V_0\times S)/\Delta_{V_0}^{\alpha|_{V_0}} &\cong (V_0\times S)/\Delta_{V_0}^{\varphi_{0,j}^{k,l}\circ\varphi},\\
	(V_0\times S)/\Delta_{V_0}^{\alpha\circ\varphi} &\cong (V_0\times S)/\Delta_{V_0}^{\varphi_{0,j}^{k,l}},\\
	(V_0\times S)/\Delta_{V_0}^{\varphi_{0,j}^{k,l}\circ c_{v^{t}}} &\cong (V_0\times S)/\Delta_{V_0}^{\varphi_{0,j}^{k,l}\circ c_{v^{t'}}\circ\varphi}\quad(t,t'\neq 0).
\end{align*}
Similarly,
\begin{align*}
	(V_0\times S)/\Delta_{\la z \ra}^{u_j^{rk}} &\cong (V_0\times S)/\Delta_{\la z^{r^{-1}}\ra}^{u_j^m},\\
	(V_0\times S)/\Delta_{\la u_0 \ra}^{u_j^{m}} &\cong (V_0\times S)/\Delta_{\la u_0\ra}^{u_j^{rl}},\\
	(V_0\times S)/\Delta_{\la z^{t}u_0\ra}^{u_j^m} &\cong (V_0\times S)/\Delta_{\la z^{r^{-1}}u_0^{s^{-1}t'}\ra}^{u_j^m}\quad(t,t'\neq 0).
\end{align*}
Note that the numerical relations observed in Remark~\ref{R:numerical rel} guarantees that the numbers of transitive $V_0$-$S$-subbisets in the above list of isomorphisms do match up.  This shows that for each $\varphi \in \Aut_\CF(V_0)$, there is an element $g\in N_G(V_0)$ which realizes $\varphi$ and such that $\ol{g}\in \Sym(J)$ respects the above list of isomorphisms of transitive $V_0$-$S$-bisets.  Schematically, $\ol{g}$ can be represented as follows:
\begin{align*}
	\overbrace{\bullet\bullet\bullet\bullet\bullet\bullet\bullet\,\bullet}^{[\alpha]}
\overbrace{(\circ\ast\ast)\cdots(\circ\ast\ast)}^{[\varphi_{0,j}^{k,l}]} |
\cdots \overbrace{(\bullet\bullet\bullet)\cdots(\bullet\bullet\bullet)}^{[\varphi_{p,j}^{k,l}]}
\overbrace{(\circ\circ\circ\ast\ast\ast\ast\ast\ast)\cdots(\circ\circ\circ\ast\ast\ast\ast\ast\ast)}^{[u_0\to u_j^m]} |\cdots \\
	\overbrace{\circ\circ\circ\circ\circ\circ\circ\,\circ}^{[\alpha]}
\overbrace{(\bullet\ast\ast)\cdots(\bullet\ast\ast)}^{[\varphi_{0,j}^{k,l}]} |
\cdots \overbrace{(\circ\circ\circ)\cdots(\circ\circ\circ)}^{[\varphi_{p,j}^{k,l}]}
\overbrace{(\bullet\bullet\bullet\ast\ast\ast\ast\ast\ast)\cdots(\bullet\bullet\bullet\ast\ast\ast\ast\ast\ast)}^{[u_0\to u_j^m]} |\cdots
\end{align*}
where $\ol{g}$ sends solid dots, circles and asterisks in the first row to solid dots, circles and asterisks in the second row, respectively, inside each compartment divided by vertical lines.  In particular, the images of $\Aut_\CF(V_i) (1\leq i\leq n)$ in $\Sym(J)$ merge left $\ol{S}$-orbits of $J$ of different sizes into a single $\ol{H}$-orbit.  Together with the transitivity of the image of $\Aut_\CF(S)$ in $\Sym(J)$ on the singleton left $S$-orbits, this shows that $\ol{H}$ is transitive on the whole index set $J$.
\end{proof}

\appendix
\section{Coefficients of characteristic idempotents} \label{S:coeff char id}

In this section, we determine some coefficients of characteristic idempotents.  The reason why we include this section in the paper is two-fold.  First, as characteristic idempotents are determined by imposing only one additional idempotency condition \cite[5.5, 5.6]{Ragnarsson2006} on virtual characteristic bisets with coefficients in $\BZ_{(p)}$, this can be done without much effort using Proposition~\ref{T:layer1} and Theorem~\ref{T:minimal biset for RV}.  More importantly, the information contained in this section can be useful for testing conjectures on the behavior of characteristic idempotents.

We follow Ragnarsson's notation \cite{Ragnarsson2006,RagnarssonStancuSFSasIdemp} in this section.  Let $\F$  be a saturated fusion system on a finite $p$-group $S$. Let $\omega = \omega_\F$ be the characteristic idempotent of $\F$.  As for characteristic bisets, let
\[
	\omega = \omega_0 + \omega_1 + \cdots,
\]
where
\[
	\omega_r = \sum_{\substack{P\leq S \\ |S:P|=p^r}}\sum_{[P,\varphi]} c_{[P,\varphi]}(\omega) [P,\varphi]^S_S.
\]
($c_{[P,\varphi]}(\omega)$ denotes the coefficient of the basis element $[P,\varphi]$ for $\omega$.)  Then the idempotency condition tells us that
\[
	\sum_{[\varphi] \in \Hom_\CF(P,S)/\sim_{(S,S)}} c_{[P,\varphi]} = 
	\begin{cases}
		1,	&\text{if $P=S$}\\
		0,	&\text{if $P<S$}	
	\end{cases}
\]
Using this, it is easy to see that
\[
	\omega_0 = c_0 \sum_{[\alpha]\in\Out_\F(S)}[S,\alpha]^S_S,
\]
where $c_0 = \frac{1}{|\Out_\F(S)|}$.

\begin{prp} \label{T:layer1-omega}
Let $\CF$ be a saturated fusion system on a finite $p$-group $S$.  Let $\{ P_i \mid 1\leq i \leq n \}$ be the set of subgroups of $S$ of index $p$.  For each $1\leq i\leq n$, let $\{ \varphi_{i,j}\colon 
P_i \to S \mid 0\leq j\leq m_i \}$ be a set of representatives of $\sim_{(S,S)}$-equivalence classes of $\CF$-morphisms with domain $P_i$ such that $\varphi_{i,0}=\iota_{P_i}$, the inclusion map $P_i\hookrightarrow S$.  Let
\begin{gather*}
	d^{(e)}_i = | \{ 1\leq j\leq m_i \mid \text{$\varphi_{i,j}$ is extendable} \} |,\\
	d^{(n)}_i = | \{ 1\leq j\leq m_i \mid \text{$\varphi_{i,j}$ is nonextendable} \} |.
\end{gather*}
Then
\[
	\omega_1 = \sum_{\substack{ i, j \text{ s.t.} \\ \varphi_{i,j} \text{ extendable}}} c^{(e)}_i [P_i,\varphi_{i,j}]^S_S + \sum_{\substack{ i, j \text{ s.t.} \\ \varphi_{i,j} \text{ nonextendable}}} c^{(n)}_i [P_i,\varphi_{i,j}]^S_S,
\]
where
\[
	c^{(e)}_i = -\frac{d^{(n)}_i}{d^{(e)}_i+pd^{(n)}_i}c_0,\quad c^{(n)}_i = \frac{d^{(e)}_i}{d^{(e)}_i+pd^{(n)}_i}c_0.
\]
\end{prp}

\begin{proof}
By Proposition~\ref{T:layer1}, we have
\[
	\omega_1 = \sum_{\substack{ i, j \text{ s.t.} \\ \varphi_{i,j} \text{ extendable}}} c_i^{(e)} [P_i,\varphi_{i,j}] +
		\sum_{\substack{ i, j \text{ s.t.} \\ \varphi_{i,j} \text{ nonextendable}}} c_i^{(n)} [P_i,\varphi_{i,j}]
\]
where $c_i^{(e)}, c_i^{(n)} \in \BZ_{(p)}$ such that
\[
	c_i^{(n)} = c_0 + pc_i^{(e)}.
\]
The coefficients of $\omega_1$ are determined by the idempotency condition:
\[
	d_i^{(e)}c_i^{(e)} + d_i^{(n)}c_i^{(n)} = 0\qquad (1\leq i\leq n).
\]
Solving these equations for $c_i^{(e)}$, we get
\[
	c_i^{(e)} = -\frac{d^{(n)}_i}{d^{(e)}_i+pd^{(n)}_i}c_0.
\]
Then
\[
	c^{(n)}_i = \frac{d^{(e)}_i}{d^{(e)}_i+pd^{(n)}_i}c_0.\qedhere
\]
\end{proof}

For the next proposition, we use the notations of \S\ref{S:minimal biset for RV}.

\begin{prp}
Let $p$ be an odd prime and let $S$ be an extraspecial group of order $p^3$ and exponent $p$.  Let $\CF$ be a saturated fusion system on $S$ such that all subgroups of $S$ of order $p^2$ are $\CF$-radical. Then
\begin{align*}
	\omega_1 =&  -\frac{c_0}{1+p} \sum_{\substack{V_i\cong_\CF V_j\\ (k,l)\in\Lambda^{e}_i}} [V_i,\psi_{i,j}^{k,l}]^S_S + \frac{c_0}{1+p} \sum_{\substack{V_i\cong_\CF V_j\\ (k,l)\in\Lambda^{n}_i}} [V_i,\varphi_{i,j}^{k,l}]^S_S,\\
	\omega_2 =& \frac{p}{p^3-1}\sum_{1\leq m\leq p-1}[z,z^m] 
	-\frac{p}{(p+1)(p^3-1)} \sum_{\substack{0\leq i\leq p\\ 1\leq m\leq p-1}}[u_i,z^m]\\ 
	&-\frac{p}{(p+1)(p^3-1)} \sum_{\substack{0\leq j\leq p\\ 1\leq m\leq p-1}}[z^m,u_j^m]
	+ \sum_{0\leq i\leq p} \left(\frac{1}{p^3-1} - \frac{r_ic_0}{p+1}\right)\sum_{\substack{V_i\cong_\CF V_j\\ 1\leq m\leq p-1}}[u_i,u_j^m]\\
	&+ \frac{1}{p^3-1} \sum_{0\leq i\leq p} \sum_{\substack{V_i\not\cong_\CF V_j\\ 1\leq m\leq p-1}}[u_i,u_j^m].
\end{align*}
where $[\xi,\zeta] := [\la \xi \ra,\varphi]$ with $\varphi \colon \la \xi\ra \to \la\zeta\ra$ such that $\varphi(\xi)=\zeta$.
\end{prp}

\begin{proof}
By Proposition~\ref{T:layer1-omega}, 
\[
	\omega_1 = \sum_{\substack{V_i\cong_\CF V_j\\ (k,l)\in\Lambda^{e}_i}} c^{(e)}_i [V_i,\psi_{i,j}^{k,l}]^S_S + \sum_{\substack{V_i\cong_\CF V_j\\ (k,l)\in\Lambda^{n}_i}} c^{(n)}_i [V_i,\varphi_{i,j}^{k,l}]^S_S,
\]
where
\[
	c^{(e)}_i = -\frac{d^{(n)}_i}{d^{(e)}_i+pd^{(n)}_i}c_0,\quad c^{(n)}_i = \frac{d^{(e)}_i}{d^{(e)}_i+pd^{(n)}_i}c_0.
\]
So we only need to compute $d_i^{(e)}$, $d_i^{(e)}$ to determine $\omega_1$.  By Lemma~\ref{T:f-number}, we have
\begin{align*}
	d_i^{(e)} &= |\{ \psi_{i,j}^{k,l} \mid V_i\cong_\CF V_j, (k,l) \in \Lambda_i^e \}| = |V_i^\CF|(p-1)r_i = \frac{1}{c_0},\\
	d_i^{(n)} &= |\{ \varphi_{i,j}^{k,l} \mid V_i\cong_\CF V_j, (k,l) \in \Lambda_i^n \}| = |V_i^\CF|(p-1)r_i = \frac{1}{c_0}.
\end{align*}
So
\[
	c_i^{(e)} = -\frac{c_0}{1+p} =: c_1,\qquad c_i^{(n)} = \frac{c_0}{1+p}.
\]

Now we determine $\omega_2$.  By the proof of Theorem~\ref{T:minimal biset for RV},
\begin{align*}
	\omega_2 &= \sum_{1\leq m\leq p-1}c_2^{(z,z^m)}[z,z^m] 
	+ \sum_{\substack{0\leq i\leq p\\ 1\leq m\leq p-1}}c_2^{(u_i,z^m)}[u_i,z^m]\\ 
	&+ \sum_{\substack{0\leq j\leq p\\ 1\leq m\leq p-1}}c_2^{(z,u_j^m)}[z,u_j^m]
	+ \sum_{\substack{0\leq i, j\leq p\\ 1\leq m\leq p-1}}c_2^{(u_i,u_j^m)}[u_i,u_j^m],
\end{align*}
where 
\begin{align*}
	c_2^{(z,z^m)} &= c_2^{(z,z)} =: c_2^{(z)},\\
	c_2^{(u_i,u_j^m)} &=
		\begin{cases}
			c_2^{(u_i,u_i)} =: c_2^{(u_i)},			&\text{if $V_i\cong_\CF V_j$}\\
			c_2^{(u_i)} + r_ic_0+pr_ic_1, 			&\text{if $V_i\not\cong_\CF V_j$},
		\end{cases}\\
	c_2^{(z,u_j^m)} &= pc_2^{(z)}+(f-r_j)\sum_{V_i\cong_\CF V_j}c_1+f\sum_{V_i\not\cong_\CF V_j}c_1,\\
	c_2^{(u_i,z^m)} &= \frac{1}{p}c_2^{(u_i)}-\frac{f-r_i}{p}c_0-(f-
r_i)c_1.
\end{align*}	
The coefficients of $\omega_2$ are determined by the idempotency condition:
\begin{align*}
	0 &= \sum_{1\leq m\leq p-1} c_2^{(z,z^m)} + \sum_{\substack{0\leq j\leq p\\ 1\leq m\leq p-1}} c_2^{(z,u_j^m)},\\
	0 &= \sum_{1\leq m\leq p-1}c_2^{(u_i,z^m)} + \sum_{\substack{0\leq j\leq p\\ 1\leq m\leq p-1}}c_2^{(u_i,u_j^m)} \qquad (0\leq i\leq p).
\end{align*}
Solving these equations, we get
\[
	c_2^{(z)} = \frac{p}{p^3-1},\qquad	c_2^{(u_i)} = \frac{1}{p^3-1} - \frac{r_i}{p+1}c_0.
\]
Then
\begin{align*}
	c_2^{(u_i,u_j^m)} &= \frac{1}{p^3-1}\qquad (\text{if $V_i\not\cong_\CF V_j$}),\\
	c_2^{(z,u_j^m)} &= c_2^{(u_i,z^m)} = -\frac{p}{(p+1)(p^3-1)}.\qedhere
\end{align*}
\end{proof}

\end{document}